\documentclass[reqno]{amsart}
\usepackage{amsmath, amssymb, amsthm, epsfig}
\usepackage[hidelinks]{hyperref}
\usepackage{latexsym}
\usepackage{url}
\usepackage[mathscr]{euscript}

\usepackage{color}
\usepackage{fullpage}
\usepackage{setspace}

\def\today{\ifcase\month\or
  January\or February\or March\or April\or May\or June\or
  July\or August\or September\or October\or November\or December\fi
  \space\number\day, \number\year}

\DeclareMathOperator{\supp}{\mathrm{supp}}

 \newtheorem{theorem}{Theorem}
  
 \newtheorem{lemma}[theorem]{Lemma} 
 
 \newtheorem{corollary}[theorem]{Corollary}
 \theoremstyle{definition}

 \theoremstyle{remark}

 \newcommand{\R}{\mathbb{R}}

 \newcommand{\Z}{\mathbb{Z}}

\newcommand{\im}{{\rm Im}\,}
\newcommand{\re}{{\rm Re}\,}

\newcommand{\norm}[1]{\left\lVert#1\right\rVert}

\newcommand{\Cplusfour}{1.17233}	
\newcommand{\oneoverCplusfour}{0.8531} 
\newcommand{\threeoverCplusfour}{2.5591}
\newcommand{\Cplusthreesixoneone}{1.1965}
\newcommand{\oneoverCplusthreesixoneone}{0.8358}

\title{Primes in arithmetic progressions and semidefinite programming}

\author[Chirre]{Andr\'es Chirre}
\address{Department of Mathematical Sciences, Norwegian University of Science and Technology, NO-7491 Trondheim, Norway}
\email{carlos.a.c.chavez@ntnu.no}

\author[Pereira J\'{u}nior]{Valdir Jos\'{e} Pereira J\'{u}nior}
\address{IMPA - Instituto Nacional de Matem\'{a}tica Pura e Aplicada - Estrada Dona Castorina, 110, Rio de Janeiro, RJ, Brazil 22460-320}
\email{valdirjosepereirajunior@gmail.com}

\author[de Laat]{David de Laat} 
\address{Delft Institute of Applied Mathematics\\
Delft University of Technology\\
P.O.\ Box 5031, 2600 GA Delft, The Netherlands}
\email{d.delaat@tudelft.nl}

\begin{document}


\allowdisplaybreaks
\numberwithin{equation}{section}
\maketitle

\begin{abstract}
Assuming the generalized Riemann hypothesis, we give asymptotic bounds on the size of intervals that contain primes from a given arithmetic progression using the approach developed by Carneiro, Milinovich and Soundararajan [Comment. Math. Helv. 94, no. 3 (2019)]. For this we extend the Guinand-Weil explicit formula over all Dirichlet characters modulo $q \geq 3$, and we reduce  the associated extremal problems to convex optimization problems that can be solved numerically via semidefinite programming.
\end{abstract}

\section{Introduction}

\subsection{Prime numbers} Denote by $\pi(x)$ the number of primes less than or equal to $x$. A classical theorem of Cram\'er \cite{cramer} states that, assuming the Riemann hypothesis (RH), for any $\alpha\geq 0$ there is a constant $c>0$ such that
\begin{align} \label{8_01_20202}
\dfrac{\pi\big(x+c\,\sqrt{x}\log x\big)-\pi(x)}{\sqrt{x}} > \alpha
\end{align}
for all sufficiently large $x$. The order of magnitude in this estimate has never been improved, and the efforts have thus been concentrated in optimizing the values of the implicit constants. Recently, Carneiro, Milinovich and Soundararajan \cite{CMS} used Fourier analysis to establish the best known values. This approach studies some Fourier optimization problems that are of the kind where one prescribes some constraints for a function and its Fourier transform, and then wants to optimize a certain quantity.

Let us define by $\mathcal{A}^{+}$ the set of even and continuous functions $F \colon \R\to\R$, with $F\in L^1(\R)$. For $1\leq A <\infty$ we write
\begin{equation} \label{17_53_08_20}
\mathcal{C}^{+}(A):=\sup_{\substack{\,\,\,\,F\in\mathcal{A}^{+}\\ F\neq 0}}\dfrac{1}{\norm{F}_1}\bigg(F(0)-A\int_{[-1,1]^c}\big(\widehat{F}\big)^{+}(t)\,dt\bigg),
\end{equation}
where we use the notation $f^{+}(x)=\max\{f(x),0\}$, $[-1,1]^c = \R \setminus [-1,1]$, and 
\[
\widehat{F}(t)=\int_{-\infty}^{\infty}F(x)e^{-2\pi ixt} \, dx.
\]
Assuming RH, \cite[Theorem 1.3]{CMS} establishes that for $\alpha\geq 0$,
\begin{align}  \label{1_20_8_04}
\inf\bigg\{c>0;\,\, \displaystyle\liminf_{x\to\infty}\,\dfrac{\pi(x+c\sqrt{x}\log x)-\pi(x)}{\sqrt{x}}> \alpha\bigg\}\leq \dfrac{(1+2\alpha)}{\mathcal{C}^{+}(36/11)}.
\end{align}
The numerical example from \cite[Eq (4.12)]{CMS} given by
\begin{equation} \label{20_00}
F(x)=-4.8\,x^2\,e^{-3.3x^2}+1.5\,x^2\,e^{-7.4x^2}+520\,x^{24}\,e^{-9.7x^2}+1.3\,e^{-2.8x^2}+0.18\,e^{-2x^2}
\end{equation}
shows that
$$
\mathcal{C}^{+}(36/11)>1.1943... >\dfrac{25}{21}.
$$
Therefore in \eqref{8_01_20202} for  $\alpha=0$ and $\alpha=1$ we can choose $c=0.8374$ and $c=2.512$, respectively. This improved the previous results established by Dudek \cite{Dud2} (see also \cite{Dud3}), who showed that for $\alpha=0$ and $\alpha=1$ we can choose $c= 1+\varepsilon$ and $c= 3+\varepsilon$, respectively, for any $\varepsilon>0$.

\subsection{Prime numbers in arithmetic progressions} Let  $q\geq 3$ and $b \geq 1$ be coprime. Denote by $\pi(x;q,b)$ the number of primes less than or equal to $x$ that are congruent to $b$ modulo $q$. Assuming the generalized Riemann hypothesis (GRH), Greni\'{e}, Molteni and Perelli \cite[Theorem 1]{GMP} stated the equivalent of the result by Cram\'er \eqref{8_01_20202} for primes in arithmetic progressions. They established that there are suitable constants $c_1, \alpha>0$ such that
\begin{align*}
\dfrac{\pi\big(x+c_1\,\varphi(q)\sqrt{x}\log x;q,b\big)-\pi(x;q,b)}{\sqrt{x}} > \alpha,
\end{align*}
for all sufficiently large $x$, where $\varphi(q)$ is Euler's function that counts the positive integers up to $q$ that are coprime to $q$. Our main goal in this paper is to show good bounds for the constant $c_1>0$.
\begin{theorem} \label{16_53_20_12}
	Assume the generalized Riemann hypothesis. Let  $q\geq 3$ and $b \geq 1$ be coprime. We have $\mathcal{C}^{+}(4) \geq \Cplusfour$, and, for any $\alpha\geq 0$,
	\begin{align*} 
	\inf\bigg\{c_1>0;\,\, \displaystyle\liminf_{x\to\infty}\dfrac{\pi\big(x+c_1\,\varphi(q)\sqrt{x}\log x;q,b\big)-\pi(x;q,b)}{\sqrt{x}}> \alpha\bigg\}\leq \dfrac{(1+2\alpha)}{\mathcal{C}^{+}(4)}< \oneoverCplusfour\,(1+2\alpha).
	\end{align*}
\end{theorem}

\medskip

 In particular, for a fixed $q\geq 3$, there is an $x_0$ sufficiently large (depending on $q$), such that for $x\geq x_0$ and all $b$ modulo $q$ there is a prime $p$ that is congruent to $b$ modulo $q$ in the interval $(x,x+\oneoverCplusfour\,\varphi(q)\sqrt{x}\log x]$. Similarly, there is an $x_0$ sufficiently large (depending on $q$), such that there are at least $\sqrt{x}$ primes which are congruent to $b$ modulo $q$ in the interval $(x,x+\threeoverCplusfour\,\varphi(q)\sqrt{x}\log x]$. This result improves asymptotically some results from a recent work by Dudek, Greni\'{e}, and Molteni \cite[Theorems 1.1-1.3]{Dud}, where they established the constants $c_1=1$ and $c_1=3$ as opposed to $c_1=\oneoverCplusfour$ and $c_1=\threeoverCplusfour$, for $\alpha=0$ and $\alpha=1$, respectively (but where also values for $x_0$ and $q_0$ such that the claim holds for $x\geq x_0$ and $q\geq q_0$ are explicitly determined).

\begin{corollary} Assume the generalized Riemann hypothesis.  Let  $q\geq 3$ and $b \geq 1$ be coprime and denote by $p_{n,q,b}$ the $n$-th prime that is congruent to $b$ modulo $q$. Then
	$$
	\limsup_{n\to\infty}\dfrac{p_{n+1,q,b}-p_{n,q,b}}{\sqrt{p_{n,q,b}}\,\,\log p_{n,q,b}}<\oneoverCplusfour\,\varphi(q).
	$$
\end{corollary}	

\medskip

\subsection{Optimized bounds}
It is mentioned in \cite{CMS} that they have found more complicated examples that do slightly better than the function $F$ given in \eqref{20_00}. Indeed, using semidefinite programming we can give a slight improvement on \cite[Theorem 1.3 and Corollary 1.4]{CMS}: We get $\mathcal C^+(36/11) \geq \Cplusthreesixoneone$, so assuming the Riemann hypothesis, for any $\alpha\geq 0$ we have in \eqref{1_20_8_04} that
\begin{align*} 
	\inf\bigg\{c>0;\,\, \displaystyle\liminf_{x\to\infty}\,\dfrac{\pi(x+c\,\sqrt{x}\log x)-\pi(x)}{\sqrt{x}}> \alpha\bigg\}<\oneoverCplusthreesixoneone\,(1+2\alpha),
\end{align*}
and
$$
\limsup_{n\to\infty}\dfrac{p_{n+1}-p_{n}}{\sqrt{p_{n}}\,\log p_{n}}<\oneoverCplusthreesixoneone.
$$
where  $p_n$ denotes the $n$-th prime.

%


\smallskip

\subsection{Strategy outline} The proof of the first inequality in Theorem~\ref{16_53_20_12} follows the ideas developed in \cite{CMS}. We will need three main ingredients: the Guinand-Weil explicit formula for the Dirichlet characters modulo $q$, the Brun-Titchmarsh inequality for primes in arithmetic progressions and the derivation of an extremal problem in Fourier analysis. We start establishing an extended version of the classical Guinand-Weil explicit formula, that contains certain sums that run over all Dirichlet characters modulo $q$. In particular, one of these sums allows us to count primes in an arithmetic progression, and we can bound many of these primes using the Brun-Titchmarsh inequality for primes in arithmetic progressions. Since many of the computations to derive the extremal problem are similar to \cite{CMS}, we will highlight the principal differences. 

For the second inequality in Theorem~\ref{16_53_20_12} we write the resulting optimization problem as a convex optimization problem over nonnegative functions. We then write these nonnegative functions as $f(x) = p(x^2) e^{-\pi x^2}$ for some polynomial $p$, as in the works of Cohn and Elkies \cite{cohn} for the sphere packing problem, and use semidefinite programming to optimize over these nonnegative functions, which is an approach employed recently for problems involving the Riemann zeta-function and other $L-$functions in \cite{andres,david}. 

\smallskip

\section{Guinand-Weil explicit formula and Brun-Titchmarsh inequality}

\subsection{Guinand-Weil explicit formula} The classical Guinand-Weil explicit formula \cite[Lemma 5]{CF} establishes the relation between the zeros of a primitive Dirichlet character modulo $q$ and the primes that are coprime to $q$. The following lemma states a version of this explicit formula that contains the sum over primitive and imprimitive Dirichlet characters modulo $q$.

\begin{lemma} \label{Guinand-weil}
Let  $q\geq 3$ and $b \geq 1$ be fixed coprime numbers. Let $h(s)$ be analytic in the strip $|\im{s}|\leq \tfrac12+\varepsilon$ for some $\varepsilon>0$, and assume that $|h(s)|\ll(1+|s|)^{-(1+\delta)}$ as $|\re{s}|\to\infty$, for some $\delta>0$. Then,
\begin{align*}
\displaystyle\sum_{\chi}\overline{\chi(b)}\displaystyle\sum_{\rho_\chi} &h\left(\frac{\rho_{\chi} - \tfrac12}{i}\right)
 = h\left(\dfrac{1}{2i}\right)+h\left(-\dfrac{1}{2i}\right)  + \frac{1}{2\pi}\displaystyle\sum_{\chi}\overline{\chi(b)}\int_{-\infty}^\infty h(u)\,{\rm Re}\,\frac{\Gamma'}{\Gamma}\bigg(\frac{1}{4}+\frac{\mathfrak{a}_\chi}{2}+\frac{iu}{2}\bigg)\,du \\
&  -\frac{\varphi(q)}{2\pi}\sum_{n\,\equiv\, b \,(\mathrm{mod}\, q)}\frac{\Lambda(n)}{\sqrt{n}}\,\widehat h\left(\frac{\log n}{2\pi}\right)-\frac{1}{2\pi}\sum_{n=2}^\infty\frac{\Lambda(n)}{\sqrt{n}}\Bigg(\displaystyle\sum_{\chi}\overline{\chi(bn)}\Bigg) \widehat h\left(\frac{-\log n}{2\pi}\right) + O\bigl(\|\widehat{h}\|_{\infty}\bigl),
\end{align*}
where $\chi$ runs over the Dirichlet characters modulo $q$ and $\rho_{\chi}$ are the non-trivial zeros of the Dirichlet $L$-function $L(s,\chi)$, $\Gamma'/\Gamma$ is the logarithmic derivative of the Gamma function, $\mathfrak{a}_\chi\in\{0,1\}$, and $\Lambda(n)$ is the Von-Mangoldt function defined to be $\log p$ if $n=p^m$ with $p$ a prime number and $m\geq 1$ an integer, and zero otherwise.  
\end{lemma}
\begin{proof}
Let $\chi$ be a primitive Dirichlet character modulo $q$. The Guinand-Weil explicit formula for $\chi$ (see \cite[Lemma 5]{CF}) states that
\begin{align} \label{22_4_6:38pm}
	\begin{split}
	\sum_{\rho_\chi} h\left(\frac{\rho_{\chi} - \tfrac12}{i}\right)= & \bigg\{\frac{\log q}{2\pi}\,\widehat{h}(0)-\frac{\log\pi}{2\pi}\,\widehat{h}(0)\bigg\} +\frac{1}{2\pi}\int_{-\infty}^\infty h(u)\,{\rm Re}\,\frac{\Gamma'}{\Gamma}\bigg(\frac{1}{4}+\frac{\mathfrak{a}_\chi}{2}+\frac{iu}{2}\bigg)\,du  \\
	& \,\,\, -\frac{1}{2\pi}\sum_{n=2}^\infty\frac{\Lambda(n)}{\sqrt{n}}\left\{\chi(n)\, \widehat h\left(\frac{\log n}{2\pi}\right)+\overline{\chi(n)}\, \widehat h\left(\frac{-\log n}{2\pi}\right)\right\},
		\end{split}
\end{align}
where the sum runs over all non-trivial zeros $\rho_\chi$ of $L(s, \chi)$, $\mathfrak{a}_\chi=0$ if $\chi(-1)=1$ and $\mathfrak{a}_\chi=1$ if $\chi(-1)=-1$. Note that
$$
\bigg|\dfrac{\log q}{2\pi}\widehat{h}(0)-\dfrac{\log \pi}{2\pi}\widehat{h}(0)\bigg|\ll \|\widehat{h}\|_{\infty}.
$$
We want to establish a similar formula as \eqref{22_4_6:38pm} for an imprimitive Dirichlet character modulo $q$. We know that each imprimitive character $\chi$ modulo $q$ is induced by a unique primitive character $\chi^{*}$ modulo $f$, with $f|q$ and $f<q$. This implies that $\chi(n)=\chi_0(n)\chi^{*}(n)$ for all $n\in\Z$, where $\chi_0(n)$ is the principal character modulo $q$, and
\begin{align} \label{22_27_21_04}
L(s,\chi)=L(s,\chi^{*})\prod_{p|q}\bigg(1-\dfrac{\chi^{*}(p)}{p^s}\bigg).
\end{align} 
If we write $\widetilde{\chi}_{0}(n)=1-\chi_{0}(n)$, then $\chi^{*}(n)=\chi(n)+\chi^{*}(n)\widetilde{\chi}_{0}(n)$. Let $\chi$ be a non-principal imprimitive character modulo $q$. Therefore, using the Guinand-Weil explicit formula for $\chi^{*}$ we get that
\begin{align} \label{22_4_6:392pm}
	\begin{split}
	\sum_{\rho_{\chi^{*}}} h\left(\frac{\rho_{\chi} - \tfrac12}{i}\right)&= \bigg\{\frac{\log f}{2\pi}\,\widehat{h}(0)-\frac{\log\pi}{2\pi}\,\widehat{h}(0)\bigg\}+\frac{1}{2\pi}\int_{-\infty}^\infty h(u)\,{\rm Re}\,\frac{\Gamma'}{\Gamma}\bigg(\frac{1}{4}+\frac{\mathfrak{a}_{\chi^{*}}}{2}+\frac{iu}{2}\bigg)\,du  \\
	& \, \, \, \, \, \, \,\,\,\,\, -\frac{1}{2\pi}\sum_{n=2}^\infty\frac{\Lambda(n)}{\sqrt{n}}\left\{\chi(n)\, \widehat h\left(\frac{\log n}{2\pi}\right)+\overline{\chi(n)}\, \widehat h\left(\frac{-\log n}{2\pi}\right)\right\}  \\
	& \, \, \, \, \, \, \,\,\,\,\, -\frac{1}{2\pi}\sum_{n=2}^\infty\frac{\Lambda(n)}{\sqrt{n}}\left\{\chi^{*}(n)\widetilde{\chi}_{0}(n)\, \widehat h\left(\frac{\log n}{2\pi}\right)+\overline{\chi^{*}(n)}\widetilde{\chi}_{0}(n)\, \widehat h\left(\frac{-\log n}{2\pi}\right)\right\}.
	\end{split}
\end{align}
Note that $\widetilde{\chi}_{0}(n)=0$ when $n$ and $q$ are coprime. Therefore the last sum can be bounded in the following form
\begin{align}\label{22_4_6:392pm2}
\Bigg|\frac{1}{2\pi}\sum_{n=2}^\infty\frac{\Lambda(n)}{\sqrt{n}}\chi^{*}(n)\widetilde{\chi}_{0}(n)\,\widehat{h}\left(\frac{\log n}{2\pi}\right)\Bigg| \ll \displaystyle\sum_{p|q, \hspace{0.05cm}k\geq 1}\dfrac{\log p}{p^{k/2}}\,\|\widehat{h}\|_{\infty}\ll \|\widehat{h}\|_{\infty}.
\end{align}
On another hand, by \eqref{22_27_21_04} we have that $L(s,\chi^*)$ and $L(s,\chi^*)$ have the same set of non-trivial zeros. Therefore, we conclude in \eqref{22_4_6:392pm} that for each imprimitive character $\chi$ modulo $q$, it follows
\begin{align}  \label{last_sum}
\begin{split}
\sum_{\rho_{\chi}} h\left(\frac{\rho_{\chi} - \tfrac12}{i}\right)&= \frac{1}{2\pi}\int_{-\infty}^\infty h(u)\,{\rm Re}\,\frac{\Gamma'}{\Gamma}\bigg(\frac{1}{4}+\frac{\mathfrak{a}_{\chi}}{2}+\frac{iu}{2}\bigg)\,du +O\bigl(\|\widehat{h}\|_{\infty}\bigl)  \\
& \, \, \, \, \, \,\,\, -\frac{1}{2\pi}\sum_{n=2}^\infty\frac{\Lambda(n)}{\sqrt{n}}\left\{\chi(n)\, \widehat h\left(\frac{\log n}{2\pi}\right)+\overline{\chi(n)}\, \widehat h\left(\frac{-\log n}{2\pi}\right)\right\},
\end{split}
\end{align}
where $\mathfrak{a}_\chi\in \{0,1\}$. Finally, for the principal character $\chi_0(n)$ we use the Guinand-Weil explicit formula for the Riemann zeta function (see \cite[Lemma 8]{CChiM}), that states that
\begin{align} \label{22_4_6:39pm}
\begin{split}
	\displaystyle\sum_{\rho}h\left(\frac{\rho-\frac12}{i}\right) & = h\left(\dfrac{1}{2i}\right)+h\left(-\dfrac{1}{2i}\right)-\dfrac{\log\pi}{2\pi}\,\widehat{h}(0)+\dfrac{1}{2\pi}\int_{-\infty}^{\infty}h(u)\,\re{\dfrac{\Gamma'}{\Gamma}\bigg(\dfrac{1}{4}+\dfrac{iu}{2}\bigg)}\,du  \\
	 &  \ \ \  -\dfrac{1}{2\pi}\displaystyle\sum_{n\geq2}\dfrac{\Lambda(n)}{\sqrt{n}}\left(\chi_0(n)\,\widehat{h}\left(\dfrac{\log n}{2\pi}\right)+\overline{\chi_0(n)}\,\widehat{h}\left(\dfrac{-\log n}{2\pi}\right)\right)\,  \\
	 &  \ \ \  -\dfrac{1}{2\pi}\displaystyle\sum_{n\geq2}\dfrac{\Lambda(n)}{\sqrt{n}}\left(\widetilde{\chi}_0(n)\,\widehat{h}\left(\dfrac{\log n}{2\pi}\right)+\overline{\widetilde{\chi}_0(n)}\,\widehat{h}\left(\dfrac{-\log n}{2\pi}\right)\right),
\end{split}
\end{align}
where the sum runs over all non-trivial zeros $\rho$ of  $\zeta(s)$, and therefore it runs over all non-trivial zeros $\rho_{\chi_0}$ of $L(s,\chi_0)$. Note that the last sum in \eqref{22_4_6:39pm} can be bounded as \eqref{22_4_6:392pm2}. Therefore, multiplying \eqref{22_4_6:38pm}, \eqref{last_sum} and \eqref{22_4_6:39pm} by $\overline{\chi(b)}$ (note that in the last case $\overline{\chi_0(b)}=1$) and summing this results to obtain the final sum over all character modulo $q$ we get (inserting the respective error terms)
\begin{align*} 
\displaystyle\sum_{\chi}\overline{\chi(b)}\displaystyle\sum_{\rho_\chi} h\left(\frac{\rho_{\chi} - \tfrac12}{i}\right) & = h\left(\dfrac{1}{2i}\right)+h\left(-\dfrac{1}{2i}\right)  + \frac{1}{2\pi}\displaystyle\sum_{\chi}\overline{\chi(b)}\int_{-\infty}^\infty h(u)\,{\rm Re}\,\frac{\Gamma'}{\Gamma}\bigg(\frac{1}{4}+\frac{\mathfrak{a}_\chi}{2}+\frac{iu}{2}\bigg)\,du  \\
& \, \, \, \, \, -\frac{1}{2\pi}\displaystyle\sum_{\chi}\sum_{n=2}^\infty\frac{\Lambda(n)}{\sqrt{n}}\left\{\overline{\chi(b)}\chi(n)\, \widehat h\left(\frac{\log n}{2\pi}\right)+\overline{\chi(b)\chi(n)}\, \widehat h\left(\frac{-\log n}{2\pi}\right)\right\}  \\
&  \, \, \, \, \,  +O\bigl(\|\widehat{h}\|_{\infty}\bigl),
\end{align*}
where the sums run over all Dirichlet characters modulo $q$, and $\mathfrak{a}_{\chi_0}=0$. Then, using Fubini's theorem and the fact that
\[
\sum_{\chi}\overline{\chi(b)}\chi(n)= \begin{cases}
\varphi(q) &   \text{if } n\equiv b \,(\mathrm{mod}\, q) \\
0 &  \text{if } n\not\equiv b \,(\mathrm{mod}\, q),
\end{cases}
\]
we obtain the desired result.
\end{proof}

\subsection{Brun-Titchmarsh inequality} We will use the following version of the Brun-Titchmarsh inequality due to Montgomery and Vaughan \cite[Theorem 2]{MV}:
\begin{align*}
\pi(x+y;q,b)-\pi(x;q,b)<\dfrac{2y}{\varphi(q)\log(y/q)}
\end{align*}
for all $x\geq1$ and $y>q$. In our case we will use this inequality in the following form: For any $\varepsilon>0$ sufficiently small and $x\geq q^{1/\varepsilon}$ we have
\begin{align} \label{16_57_20_12}
\pi(x+\sqrt{x};q,b)-\pi(x;q,b)<\frac{4}{(1-2\varepsilon)\,\varphi(q)}\dfrac{\sqrt{x}}{\log x}.
\end{align}

\medskip

\section{Proof of Theorem \ref{16_53_20_12}: First part}
We follow the idea developed in \cite[Section 5]{CMS}. We start by fixing $q\geq 3$ and $b\geq 1$ coprime and assuming GRH. Also, we fix an even and bandlimited Schwartz function $F:\R \to\R$ such that $F(0)>0$ and $\supp(\widehat{F})\subset [-N,N]$ for some parameter $N\geq 1$. Therefore, $F$ extends to an entire function, and using the Phragm\'{e}n-Lindel\"{o}f principle the hypotheses of Lemma~\ref{Guinand-weil} are satisfied. Throughout this proof, the error terms can depend on $q$, $b$, and $F$. Let $0<\Delta\leq 1$ and $1<a$ be free parameters (to be chosen later) such that
	\begin{equation} \label{5:46pm_15_8}
2\pi\Delta N\leq \log a.
\end{equation}
We need to have in mind that $a\to\infty$ and $\Delta\to 0$. Considering the  function $f(z)=\Delta F(\Delta z)$ we have supp$(\widehat{f})\subset[-\Delta N,\Delta N]$. Applying Lemma \ref{Guinand-weil} for the function $h(z)=f(z)a^{iz}$ we obtain
\begin{align} \label{UsingGW}
\begin{split}
\displaystyle\sum_{\chi}\overline{\chi(b)}\displaystyle\sum_{\gamma_{\chi}} h(\gamma_{\chi})
& = \bigg\{h\left(\dfrac{1}{2i}\right)+h\left(-\dfrac{1}{2i}\right)\bigg\}  + \frac{1}{2\pi}\displaystyle\sum_{\chi}\overline{\chi(b)}\int_{-\infty}^\infty h(u)\,{\rm Re}\,\frac{\Gamma'}{\Gamma}\bigg(\frac{1}{4}+\frac{\mathfrak{a}_\chi}{2}+\frac{iu}{2}\bigg)\,du  \\
& \, \, \, \, \, \,\,\,\,\,\, -\frac{\varphi(q)}{2\pi}\sum_{n\,\equiv\, b \, (\mathrm{mod}\, q)}\frac{\Lambda(n)}{\sqrt{n}}\,\widehat h\left(\frac{\log n}{2\pi}\right)-\frac{1}{2\pi}\sum_{n=2}^\infty\frac{\Lambda(n)}{\sqrt{n}}\Bigg(\displaystyle\sum_{\chi}\overline{\chi(bn)}\Bigg) \widehat h\left(\frac{-\log n}{2\pi}\right) \\
&  \, \, \, \, \, \,\,\,\,\,\,+ O\bigl(\|\widehat{h}\|_{\infty}\bigl),
\end{split}
\end{align}
where $\gamma_{\chi}$ are the imaginary parts of the non-trivial zeros of $L(s,\chi)$. We start by estimating some terms on the right-hand side of \eqref{UsingGW}. Using the estimate in \cite[p. 553]{CMS} we get
\begin{align} \label{R1}
h\left(\dfrac{1}{2i}\right)+h\left(-\dfrac{1}{2i}\right)=\Delta F(0)\big(\sqrt{a}+\sqrt{a^{-1}}\big)+O\big(\Delta^2\sqrt{a}\big).
\end{align}
Using Stirling's formula with a minor adjustment to the estimate in \cite[p. 554]{CMS} we have
$$
\int_{-\infty}^\infty h(u)\,{\rm Re}\,\frac{\Gamma'}{\Gamma}\left(\frac{1}{4}+\frac{\mathfrak{a}_\chi}{2}+\frac{iu}{2}\right)\,du = O(1).
$$
Therefore,
\begin{align} \label{R2}
\frac{1}{2\pi}\displaystyle\sum_{\chi}\overline{\chi(b)}\int_{-\infty}^\infty h(u)\,{\rm Re}\,\frac{\Gamma'}{\Gamma}\left(\frac{1}{4}+\frac{\mathfrak{a}_\chi}{2}+\frac{iu}{2}\right)\,du = O(1).
\end{align}
Also, note that by \eqref{5:46pm_15_8} we have for $n\geq 2$,
$$
\widehat h\left(\frac{-\log n}{2\pi}\right)=0.
$$
This implies that
\begin{align} \label{R3}
\frac{1}{2\pi}\sum_{n=2}^\infty\frac{\Lambda(n)}{\sqrt{n}}\Bigg(\displaystyle\sum_{\chi}\overline{\chi(bn)}\Bigg) \widehat h\left(\frac{-\log n}{2\pi}\right)=0.
\end{align}
Inserting \eqref{R1}, \eqref{R2} and \eqref{R3} into \eqref{UsingGW}, it follows that
\begin{align*}
\displaystyle\sum_{\chi}\overline{\chi(b)}\displaystyle\sum_{\gamma_\chi} h(\gamma_\chi)= \Delta F(0)\big(\sqrt{a}+\sqrt{a^{-1}}\big)+O\big(\Delta^2\sqrt{a}\big)  -\frac{\varphi(q)}{2\pi}\sum_{n\,\equiv\, b \, (\mathrm{mod} \, q)}\frac{\Lambda(n)}{\sqrt{n}}\,\widehat h\left(\frac{\log n}{2\pi}\right)+ O(1).
\end{align*}
Then,
\begin{align} \label{UsingGW2}
\begin{split}
\Delta F(0)\,\sqrt{a}\leq & \displaystyle\sum_{\chi}\displaystyle\sum_{\gamma_\chi} \big|h(\gamma_\chi)\big| + \frac{\varphi(q)}{2\pi}\sum_{n\,\equiv\, b \, (\mathrm{mod}\, q)}\frac{\Lambda(n)}{\sqrt{n}}\,(\widehat h)^+\left(\frac{\log n}{2\pi}\right)
+O\big(\Delta^2\sqrt{a}\big)+ O(1).
\end{split}
\end{align}
Let us  estimate the terms on the right-hand side of \eqref{UsingGW2}. We recall that for each primitive Dirichlet character modulo $q$, we have the formula \cite[Chapter 16]{D}
 $$
N(T,\chi)=\dfrac{T}{\pi}\log\bigg(\dfrac{qT}{2\pi}\bigg) - \dfrac{T}{\pi} + O(\log T+\log q),
 $$
 where $N(T,\chi)$ denotes the number of zeros of $L(s,\chi)$ in the rectangle $0<\sigma<1$ and $|\gamma|\leq T$. Using integration by parts as in \cite[Eq. (5.4)]{CMS} we obtain for each primitive Dirichlet character modulo $q$ that
\begin{align} \label{17_27_20_123}
\displaystyle\sum_{\gamma_\chi} |h(\gamma_\chi)| & = \frac{\log(1/2\pi\Delta)}{2\pi}\norm{F}_1 +O(1).
\end{align}
Note that the main term in the above expression is independent of $q$. Then, recalling the relation \eqref{22_27_21_04}, we have that \eqref{17_27_20_123} holds for each non-principal imprimitive character modulo $q$. For the case of the principal character $\chi_0$, we use the estimate for the zeros of the Riemann zeta-function (see \cite[Eq. (5.4)]{CMS}). Therefore, considering that the number of Dirichlet characters modulo $q$ is $\varphi(q)$ we conclude that
\begin{align} \label{17_27_20_12}
\displaystyle\sum_{\chi}\displaystyle\sum_{\gamma_\chi} \big|h(\gamma_\chi)\big| = \varphi(q)\,\frac{\log(1/2\pi\Delta)}{2\pi}\norm{F}_1 +O(1).
\end{align}
Now, we want to bound the second sum on the right-hand side of \eqref{UsingGW2}. Using the relation between the functions $h$ and $F$, this sum is
\begin{align} \label{19_49_08_20}
\sum_{n\,\equiv\, b \, (\mathrm{mod} \, q)}\frac{\Lambda(n)}{\sqrt{n}}\,(\widehat F)^+\left(\frac{\log(n/a)}{2\pi\Delta}\right).
\end{align}
Fix $\alpha\geq 0$ and assume that $c_1>0$ is a fixed constant such that
$$
\displaystyle\liminf_{x\to\infty}\,\,\dfrac{\pi\big(x+c_1\,\varphi(q)\sqrt{x}\log x;q,b\big)-\pi(x;q,b)}{\sqrt{x}}\leq \alpha.
$$
This implies that for $\varepsilon>0$, there exists a sequence of $x\to\infty$ such that
\begin{align*}
\dfrac{\pi\big(x+c_1\,\varphi(q)\sqrt{x}\log x;q,b\big)-\pi(x;q,b)}{\sqrt{x}}< \alpha+\varepsilon.
\end{align*}
For each $x$ in the sequence, we choose $a$ and $\Delta$ such that
\begin{align*}
\big[x,x+c_1\,\varphi(q)\sqrt{x}\log x\big]=\Big[ae^{-2\pi\Delta},ae^{2\pi\Delta}\Big]
\end{align*}	
and this implies that (see \cite[Eq. (5.7)-(5.8)]{CMS})
\begin{align*}
4\pi\Delta=c_1\,\varphi(q)\dfrac{\log x}{\sqrt{x}} + O\bigg(\dfrac{\log^2x}{x}\bigg),
\end{align*}
and
\begin{align*}
a=x+O(\sqrt{x}\log x). 
\end{align*}
Since $\supp(\widehat{F})\subset [-N,N]$, the sum in \eqref {19_49_08_20} runs over $ae^{-2\pi\Delta N}\leq n\leq ae^{2\pi\Delta N}$ with $n\equiv b \,(\mathrm{mod}\, q)$. Then the contribution of the prime powers $n=p^k$ with $n\equiv b \,(\mathrm{mod}\, q)$ in that interval is $O(1)$. The contribution of the (at most) $(\alpha+\varepsilon)\sqrt{x}$ primes in the interval $(x,x+c_1\,\varphi(q)\sqrt{x}\log x]=(ae^{-2\pi\Delta},ae^{2\pi\Delta}]$ to the sum \eqref{19_49_08_20} is bounded above by (using that $(\widehat{F})^{+}(t)\leq\|F\|_1$)
$$\|F\|_1\displaystyle\sum_{\substack{p\in(ae^{-2\pi\Delta},ae^{2\pi\Delta}]\\ p\,\equiv \,b \,(\mathrm{mod}\, q) }}\dfrac{\log p}{\sqrt{p}}\leq\|F\|_1(\alpha+\varepsilon)\sqrt{x}\,\,\dfrac{\log x}{\sqrt{x}}=\|F\|_1(\alpha+\varepsilon)\log x.$$
Finally, to estimate the contribution of the primes in the intervals $[ae^{-2\pi\Delta N},ae^{-2\pi\Delta}]$ and $[ae^{2\pi\Delta},ae^{2\pi\Delta N}]$ we use the Brun-Titchmarsh inequality \eqref{16_57_20_12}. We also need the following estimate: For $g\in C^1([a,b])$ we have
\begin{equation} \label{22_8_8_20}
0\leq S(g^+,P)-\displaystyle\int_a^bg^+(t)\,dt\leq \delta(b-a)\sup_{x\in [a,b]}|g^{\prime}(x)|,
\end{equation}
where $P$ is a partition of $[a,b]$ of norm at most $\delta$ and $S(g^+,P)$ is the upper Riemann sum of the function $g^+$ and the partition $P$. We apply \eqref{22_8_8_20} with the function
$$g(t)=\widehat{F}\bigg(\frac{\log (t/a)}{2\pi\Delta}\bigg),$$
and the partition $P=\{x_0<...<x_J\}$ that cover the interval $[ae^{2\pi\Delta},ae^{2\pi\Delta N}]\subset \cup_{j=0}^{J-1}[x_j,x_{j+1}]$, with $x_0=ae^{2\pi\Delta}$, $x_{j+1}=x_j+\sqrt{x_j}$. If we define $M_j=\sup\{g^+(x) : x\in[x_j.x_{j+1}]\}$, then $S(g^+,P)=\sum_{j=0}^{J-1}M_j\sqrt{x_j}$. Therefore, by \eqref{16_57_20_12} and \eqref{22_8_8_20} it follows
\begin{align}  \label{23_37_09_20}
\displaystyle\sum_{\substack{1\leq\frac{\log p/a}{2\pi\Delta}\leq N \\ p\,\equiv\, b \,(\mathrm{mod} \, q)}}&\dfrac{\log p}{\sqrt{p}}\,(\widehat{F})^+\bigg(\frac{\log (p/a)}{2\pi\Delta}\bigg) \nonumber \\
&\leq\displaystyle\sum_{j=0}^{J-1}\bigg(\dfrac{\log x_j}{\sqrt{x_j}}M_j\bigg)\dfrac{4\sqrt{x_j}}{(1-2\varepsilon)\,\varphi(q)\log x_j}  \leq \dfrac{4}{(1-2\varepsilon)\,\varphi(q)}\Bigg(\dfrac{1}{\sqrt{a}}\displaystyle\sum_{j=0}^{J-1}M_j\sqrt{x_j}\Bigg) \nonumber \\
& = \dfrac{4}{(1-2\varepsilon)\,\varphi(q)}\Bigg(\dfrac{1}{\sqrt{a}}\displaystyle\int_{x_0}^{x_J}(\widehat{F})^+\bigg(\frac{\log (t/a)}{2\pi\Delta}\bigg)\,dt +\dfrac{1}{\sqrt{a}}\Bigg(S(g^+,P)-\displaystyle\int_{x_0}^{x_J}(\widehat{F})^+\bigg(\frac{\log (t/a)}{2\pi\Delta}\bigg)\,dt\Bigg)\Bigg) \nonumber  \\
& \leq \dfrac{4}{(1-2\varepsilon)\,\varphi(q)}\Bigg(\sqrt{a}(2\pi\Delta)\displaystyle\int_1^{N}(\widehat{F})^+(t)\,e^{2\pi\Delta t}\,dt +O\Bigg(\dfrac{1}{2\pi\Delta}(e^{2\pi\Delta N}-e^{2\pi\Delta})\Bigg)\Bigg) \nonumber  \\
& \leq \dfrac{4\,\sqrt{a}\,(2\pi\Delta)}{(1-2\varepsilon)\,\varphi(q)}\,\displaystyle\int_1^{N}(\widehat{F})^+(t)\,dt+O(1),
\end{align}
where we used that $0\leq e^x-1\leq 2x$ for $0\leq x\leq 1$. We treat the other interval in a similar way, obtaining
\begin{align} \label{23_389_09_20}
\displaystyle\sum_{\substack{-N\leq \frac{\log p/a}{2\pi\Delta}\leq -1 \\ p\,\equiv\, b \,(\mathrm{mod}\, q)}}&\dfrac{\log p}{\sqrt{p}}\,(\widehat{F})^+\bigg(\frac{\log (p/a)}{2\pi\Delta}\bigg)  \leq \dfrac{4\,\sqrt{a}\,(2\pi\Delta)}{(1-2\varepsilon)\,\varphi(q)}\,\displaystyle\int_{-N}^{-1}(\widehat{F})^+(t)\,dt+O(1).
\end{align}
Combining \eqref{23_37_09_20} and \eqref{23_389_09_20} we obtain






\begin{align*}
\displaystyle\sum_{\substack{1<\left|\frac{\log(p/a)}{2\pi\Delta}\right|\leq N \\ p\,\equiv\, b \, (\mathrm{mod}\, q)}}\dfrac{\log p}{\sqrt{p}}\,(\widehat{F})^+\bigg(\dfrac{\log(p/a)}{2\pi\Delta}\bigg) &  \leq\frac{4\,\sqrt{a}(2\pi\Delta)}{(1-2\varepsilon)\varphi(q)}\int_{[-1,1]^c}(\widehat{F})^+(t)\,dt + O(1).
\end{align*}
Therefore, grouping the previous estimates, we conclude that
\begin{align} \label{17_45_20_12}
\sum_{n\,\equiv\, b \,(\mathrm{mod}\, q)}\frac{\Lambda(n)}{\sqrt{n}}\,(\widehat h)^+\left(\frac{\log n}{2\pi}\right)\leq \norm{F}_1(\alpha+\varepsilon)\log x + \frac{4\,\sqrt{a}\,(2\pi\Delta)}{(1-2\varepsilon)\varphi(q)}\int_{[-1,1]^c}(\widehat{F})^{+}(t)\,dt+ O(1).
\end{align}
Then, inserting the estimates \eqref{17_27_20_12} and \eqref{17_45_20_12} in \eqref{UsingGW2} it follows that
\begin{align*}
\begin{split}
\Delta\sqrt{a}\bigg(F(0)-\frac{4}{(1-2\varepsilon)}\int_{[-1,1]^c}(\widehat{F})^{+}(t)\,dt\bigg) \leq & \dfrac{\varphi(q)}{2\pi}\norm{F}_1(\alpha+\varepsilon)\log x + \varphi(q)\frac{\log(1/2\pi\Delta)}{2\pi}\norm{F}_1 +O(1).
\end{split}
\end{align*}
Sending $x\to\infty$ along the sequence we obtain that
$$
c_1\leq (1+2\alpha+2\varepsilon)\dfrac{\norm{F}_1}{F(0)-\frac{4}{(1-2\varepsilon)}\int_{[-1,1]^c}(\widehat{F})^{+}(t)\,dt},
$$
and with $\varepsilon\to 0$ we get
$$
c_1\leq (1+2\alpha)\dfrac{\norm{F}_1}{F(0)-4\int_{[-1,1]^c}(\widehat{F})^{+}(t)\,dt}.
$$
Finally, we recall that in \cite[Subsection 4.1]{CMS} it is showed that to find the sharp constants $\mathcal{C}^{+}(A)$ in \eqref{17_53_08_20}, without loss of generality we can restrict to the subset of $\mathcal{A}^{+}$ such that $\widehat{F} \in C_{c}^{\infty}(\R)$. Therefore, we conclude
\begin{align*}
\inf\bigg\{c_1>0;\,\, \displaystyle\liminf_{x\to\infty}\dfrac{\pi\big(x+c_1\,\varphi(q)\sqrt{x}\log x;q,b\big)-\pi(x;q,b)}{\sqrt{x}}> \alpha\bigg\}\leq \dfrac{(1+2\alpha)}{\mathcal{C}^{+}(4)}.
\end{align*}

\section{Numerically optimizing the bounds}

We first reformulate \eqref{17_53_08_20} as a convex optimization problem:
\begin{lemma}
	Let $\mathcal F$ be the set of tuples $(f_1,\ldots,f_4)$ with $f_1,\ldots,f_4 \in L^1(\R)$ even, nonnegative, continuous functions such that $\widehat f_1(0) + \widehat f_2(0) = 1$ and $f_1-f_2 = \widehat f_3 - \widehat f_4$.
	For $A \geq 1$ we have
	\[
	\mathcal C^+(A) = \sup_{(f_1,f_2,f_3,f_4) \in \mathcal F}\bigg( f_1(0) - f_2(0) - A \int_{[-1,1]^c} f_3(t)\,dt\bigg).
	\]
\end{lemma}
\begin{proof}
	Given $(f_1,f_2,f_3,f_4) \in \mathcal F$ we set $F := f_1-f_2 = \widehat f_3 - \widehat f_4$. Then,
	\[
	\|F\|_1 = \|f_1-f_2\|_1 \leq \|f_1\|_1 + \|f_2\|_1 = \widehat f_1(0) + \widehat f_2(0) = 1.
	\]
	Further, $F(0) = f_1(0) - f_2(0)$ and
	$
	(\widehat F)^+(x) = (f_3 - f_4)^+(x) \leq f_3(x).
	$
	This shows
	\[
	\mathcal C^+(A) \ge \sup_{(f_1,f_2,f_3,f_4) \in \mathcal F} \bigg(f_1(0) - f_2(0) - A \int_{[-1,1]^c} f_3(t)\,dt\bigg).
	\]
	
	On the other hand, for $F \in \mathcal A^+$ with $F \neq 0$, we define $f_1 := F^+$, $f_2 := F^-$, $f_3 := (\widehat F)^+$, and $f_4 := (\widehat F)^-$, where $f_i^-(x) = \min\{f_i(x),0\}$. Then $\|F\|_1 = \widehat f_1(0) + \widehat f_2(0)$, $F(0) = f_1(0) - f_2(0)$, and $(\widehat F)^+ = f_3$, which shows the other inequality.
\end{proof}

To find good functions we restrict to functions $f_i$ of the form $f_i(x) = p_i(x^2) e^{-\pi x^2}$ where $p_i$ is a polynomial. We then use the sum-of-squares characterization
\[
p_i(u) = v_d(u)^{\sf T} Q_i v_d(u) + u\, v_{d-1}(u)^{\sf T} R_i v_{d-1}(u),
\]
where $Q_i$ and $R_i$ are positive semidefinite matrices, and $v_k(u)$ a vector whose entries form a basis for the polynomials of degree at most $k$. This forces $p_i$ to be nonnegative on $[0,\infty)$, and each polynomial that is nonnegative on $[0,\infty)$ is of this form. For the numerical conditioning a good choice for the basis is
\[
v_k(u) = \big(L_0^{-1/2}(\pi u), \ldots, L_k^{-1/2}(\pi u)\big),
\]
where $L_i^{-1/2}$ is the Laguerre polynomial of degree $i$ with parameter $-1/2$.

The conditions $\widehat f_1(0) + \widehat f_2(0) = 1$ and  $f_1-f_2 = \widehat f_3 - \widehat f_4$ are linear in the entries of the matrices $Q_i$ and $R_i$. A numerically stable way of enforcing these constraints is to first compute $\widehat f_3 - \widehat f_4$ by using that the Fourier transform of
$|x|^{2k} e^{-\pi |x|^2}$ is $k!/\pi^k L_k^{-1/2}(\pi |x|^2) e^{-\pi |x|^2}$, and then express $f_1-f_2 -(\widehat f_3 - \widehat f_4)$ in the Laguerre basis and equating all coefficients to zero. The linear objective
\[
f_1(0) - f_2(0) - A \int_{[-1,1]^c} f_3(t)\,dt
\]
can be expressed explicitly as a linear functional in the entries of $Q_i$ and $R_i$, for $i=1,2,3$, using the identity
\[
\int x^m e^{-\pi x^2}\, dx = -\frac{1}{2\pi^{m/2+1/2}} \Gamma\left(\frac{m+1}{2}, \pi x^2\right),
\]
where $\Gamma$ is the upper incomplete gamma function
This reduces the optimization problem to  a semidefinite program that can be optimized with a numerical semidefinite programming solver.

The main issue now is to get a rigorous bound from the numerical output, for which we adapt the approach from \cite{Lo}. Instead of optimizing over $Q_i$ and $R_i$ we fix $\varepsilon = 10^{-20}$ and set $Q_i = \widetilde Q_i + \varepsilon I$ and $R_i = \widetilde R_i + \varepsilon I$. We then optimize over the positive semidefinite matrices $\widetilde Q_i$ and $\widetilde R_i$. In this way the matrices $Q_i$ and $R_i$ will be positive semidefinite even if the matrices $\widetilde Q_i$ and $\widetilde R_i$ computed by the solver have slightly negative eigenvalues. We then use ball arithmetic (using the Arb library \cite{J}) to verify rigorously that all matrices $Q_i$ and $R_i$ are indeed positive semidefinite, and compute a rigorous lower bound $b$ on the smallest eigenvalue of $Q_4$. Then we compute a rigorous upper bound $B$ on the largest (in absolute value) coefficient of $(\widehat f_1- \widehat f_2) - (f_3 - f_4)$ in the basis given by the diagonal and upper diagonal of the matrix $v_d(x) v_d(x)^{\sf T}$. Let $Q_4'$ be the symmetric tridiagonal matrix such that
\[
(\widehat f_1- \widehat f_2) - (f_3 - f_4) = v_d(x)^{\sf T} Q_4'\, v_d(x).
\]
Since $\lambda_\mathrm{min}(Q_4+Q_4') \geq \lambda_\mathrm{min}(Q_4) + \lambda_\mathrm{min}(Q_4')$ and $\lambda_\mathrm{min}(Q_4') \geq -2B$ by the Gershgorin circle theorem, the matrix $Q_4+Q_4'$ is positive semidefinite if $b \geq 2B$. We now verify this inequality and replace $Q_4$ by $Q_4+Q_4'$. Then the identity $(\widehat f_1- \widehat f_2) - (f_3 - f_4) = 0$, and thus the identity $f_1-f_2 = (\widehat f_3 - \widehat f_4)$, holds exactly. Since
\begin{equation}\label{eq:altobj}
	\frac{1}{\widehat f_1(0) + \widehat f_2(0)} \left(f_1(0) - f_2(0) - A \int_{[-1,1]^c} f_3(t)\,dt\right)
\end{equation}
does not depend on $Q_4$, an upper bound on \eqref{eq:altobj}, again computed using ball arithmetic, gives  a rigorous upper bound on $\mathcal C^+(A)$. The matrices $Q_i$ and $R_i$ (computed with degree $d=90$) and a simple script using Julia/Nemo/Arb \cite{Julia,FHHJ,J} to perform  this verification procedure (as well as a script to setup and solve the semidefinite programs using sdpa-gmp \cite{Nakata}) are stored at the link \url{https://arxiv.org/abs/2005.02393} as ancillary files.

\bigskip

\section*{Acknowledgements}
We would like to Emanuel Carneiro and Micah Milinovich for their insightful comments. AC was supported by Grant 275113 of the Research Council of Norway. VJPJ was supported by FAPERJ-Brazil.


\bigskip

\begin{thebibliography}{9999}
	
	\bibitem{Julia}
	Jeff Bezanson, Alan Edelman, Stefan Karpinski, and Viral~B. Shah.
	\newblock Julia: a fresh approach to numerical computing.
	\newblock {\em SIAM Rev.}, 59(1):65--98, 2017.

	\bibitem{CChiM} E. Carneiro, A. Chirre and M. B. Milinovich,
	\newblock Bandlimited approximations and estimates for the Riemann zeta-function,
	\newblock Publ. Mat. 63 (2019), no. 2, 601--661.
	
	
	\bibitem{CF} E. Carneiro and R. Finder,
	\newblock On the argument of $L-$functions,
\newblock Bull. Braz. Math. Soc. vol. 46, no. 4 (2015), 601--620.
	
	
	
	
	\bibitem{CMS} E. Carneiro, M. B. Milinovich and K. Soundararajan,
	\newblock Fourier optimization and prime gaps,
	\newblock Comment. Math. Helv. 94, no. 3 (2019), 533--568.
	
	
	\bibitem{andres}
	A. Chirre, F. Gon\c{c}alves and D. de Laat,
	\newblock Pair correlation estimates for the zeros of the zeta-function via semidefinite programming.
	\newblock To appear in Adv. Math.
	
	\bibitem{cohn}
	\newblock H. Cohn and N. Elkies,
	\newblock New upper bounds on sphere packings I,
	\newblock Ann. of Math. (2) 157 (2003), no. 2, 689--714.
	
	
			
	\bibitem{cramer} H. Cram\'er,
	\newblock Some theorems concerning prime numbers,
	\newblock Ark. Mat. Astr. Fys. 15 (1920), 1--33.
	
	
	
	\bibitem{D}
	H. Davenport,
	\newblock {\it Multiplicative number theory},
	\newblock Second edition, Graduate Texts in Mathematics 74, Springer-Verlag, New
	York (1980).
	
	
	
	
	
			\bibitem{Dud2}
	A. Dudek,
	\newblock On the Riemann hypothesis and the difference between primes,
	\newblock Int. J. Number Theory 11 (2015),	no. 3, 771--778.
	
	
	\bibitem{Dud3}
	A. Dudek, L. Greni\'{e}, G. Molteni,
	\newblock Primes in explicit short intervals on RH, 
	\newblock Int. J. Number Theory 12 (2016), no. 5, 1391--1407.
	
	
	
	\bibitem{Dud}
	A. Dudek, L. Greni\'{e}, G. Molteni,
	\newblock Explicit short intervals for primes in arithmetic progressions on GRH,
	\newblock Int. J. Number Theory, Vol. 15, no. 4 (2019) 825--862.
	
	\bibitem{FHHJ}
	C.~Fieker, W.~Hart, T.~Hofmann and F.~Johansson,
	\newblock{Nemo/Hecke: Computer Algebra and Number Theory Packages for the Julia Programming Language},
	\newblock{Proceedings of ISSAC '17 (2017), 157--164.}
	
	\bibitem{GMP}
    L. Greni\'{e}, G. Molteni and A. Perelli,
    \newblock Primes and prime ideals in short intervals,	
	\newblock  Mathematika 63 (2017), no. 2, 364--371.	
	
	\bibitem{J}
	F. Johansson,
	\newblock Arb: efficient arbitrary-precision midpoint-radius interval arithmetic,
	\newblock IEEE Transactions on Computers, 66 (2017), no. 8, 1281--1292.
	
	\bibitem{david}
	D. de Laat, L. Rolen, Z. Tripp and I. Wagner,
	\newblock Pair correlation for Dedekind zeta functions of abelian extensions,
	\newblock Preprint.
	
	\bibitem{Lo}
	J. L\"ofberg,
	\newblock Pre- and post-processing sums-of-squares programs in practice
	\newblock IEEE Trans.~Automat.~Control 54 (2009), 1007–1011.
	
	\bibitem{MV}
	H. L. Montgomery and R. C. Vaughan,
	\newblock The large sieve,
	\newblock Mathematika, 20 (1973) 119--134.
	
	\bibitem{Nakata}
	M.~Nakata,
	\newblock A numerical evaluation of highly accurate multiple-precision arithmetic version of semidefinite programming solver: SDPA-GMP,-QD and-DD,
	\newblock In Computer-Aided Control System Design (CACSD), 2010 IEEE International Symposium on, pages 29–34. IEEE, 2010.
	
\end{thebibliography}
\end{document}